\documentclass[12pt]{article}

\makeatletter
\let\@fnsymbol\@arabic
\makeatother

\usepackage{amsfonts,amsmath,amsthm,amssymb,graphicx,epsfig,makeidx,latexsym}

\newtheorem{theorem}{Theorem}[section]
\newtheorem{prop}{Proposition}[section]

\newtheorem{remark}{Remark}[section]
\newtheorem{corr}{Corollary}[section]

\newcommand{\E}{{\mathbb E}}
\newcommand{\RR}{{\mathbb R}}

\newcommand {\PP}{{\mathbb P}}
\newcommand{\sss}{\scriptscriptstyle}

\begin{document}

\title{Routing on trees}\parskip=5pt plus1pt minus1pt \parindent=0pt 
\author{Maria Deijfen\thanks{Stockholm University, Sweden; {\tt mia@math.su.se}} \and Nina Gantert\thanks{Technische Universit\"{a}t M\"{u}nchen, Germany ; {\tt gantert@ma.tum.de}}}
\date{July 9, 2014}
\maketitle

\begin{abstract}
\noindent We consider three different schemes for signal routing on a tree. The vertices of the tree represent transceivers that can transmit and receive signals, and are equipped with i.i.d.\ weights representing the strength of the transceivers. The edges of the tree are also equipped with i.i.d.\ weights, representing the costs for passing the edges. For each one of our schemes, we derive sharp conditions on the distributions of the vertex weights and the edge weights that determine when the root can transmit a signal over arbitrarily large distances.

\noindent
\vspace{0.3cm}


\vspace{0.2cm}

\end{abstract}

\section{Introduction}

Let $\mathcal{T}$ be a rooted infinite $m$-ary tree and assign i.i.d.\ weights $\{R_x\}$ to the vertices of $\mathcal{T}$ and i.i.d.\ weights $\{C_e\}$ to the edges. Assume that $\{R_x\}$  is independent of $\{C_e\}$. We think of the vertices as representing transceivers that can receive and transmit signals. The vertex weights represent the strength or range of the transceivers and the edge weights represent the cost or resistance when traversing the edges. We study three different schemes for signal routing in $\mathcal{T}$ and, for each of these schemes, we investigate when the root can transmit a signal over arbitrarily large distances. More specifically, write $O$ for the set of vertices that are reached by a signal transmitted by the root, and say that a scheme can transmit indefinitely if $|O|=\infty$ with positive probability. Our main results are sharp conditions on the distributions of $R$ and $C$ that determine when the respective routing schemes can transmit indefinitely. Here and throughout the paper, $R$ and $C$ denote random variables with the laws of $R_x$ and $C_e$, respectively.

Write $\Gamma_{x,y}$ for the path between the vertices $x$ and $y$ in $\mathcal{T}$, and write $y>x$ if $y$ is located in the subtree below $x$ in $\mathcal{T}$ (so that $y$ is hence further away from the root than $x$). For each vertex $x$, let $\Lambda_x$ be the set of all vertices $y$ in the subtree below $x$ for which the total cost of the path from $x$ to $y$ does not exceed the range of $x$, that is,
$$
\Lambda_x=\Big\{y>x:\sum_{e\in\Gamma_{x,y}}C_e\leq R_x\Big\}.
$$
We say that the vertices in $\Lambda_x$ are within the range of $x$. The schemes that we will consider are now defined as follows.

\noindent\textbf{Complete routing.} The root 0 first transmits the signal to all vertices in $\Lambda_0$. In the next step, each vertex $x\in \Lambda_0$ forwards the signal to all vertices in $\Lambda_x$, and the signal is then forwarded according to the same rule by each new vertex that is reached by it. Note that edges leading back towards the root are not used in the forwarding process, that is, the transceivers do not forward the signal through the same edge that the signal arrived from. This simplifies the analysis since it implies that whether a signal reaches a vertex $y$ or not is determined only by the configuration on the path between 0 and $y$.

\noindent\textbf{Boundary routing.} For a connected subset $\Omega$ of the vertices in $\mathcal{T}$, with $0\in \Omega$, let $\partial \Omega$ denote the set of vertices in $\Omega$ that have at least one child that is not in $\Omega$. The transmission is initiated in that the root 0 transmits the signal to all vertices in $\Lambda_0$, and the signal is then forwarded stepwise: If the set of vertices that have received the signal after a certain step is $\Omega$, then, in the next step, the signal is forwarded by each $x\in \partial \Omega$ to all vertices $y$ in $\Lambda_x$ such that the path between $x$ and $y$ (excluding $x$) contains only vertices in $\Omega^c$. The difference compared to complete routing is hence that only vertices with neighbors that have not yet heard the signal forward the signal and then only in the direction of these un-informed neighbors.

\noindent\textbf{Augmented routing.} For a vertex $x$ at level $k$ in $\mathcal{T}$, write $0=x_0,\ldots,x_k=x$ for the path from the root to $x$. In the last scheme, when a signal traverses an edge, its strength is reduced by the cost of the edge, and when it passes a transceiver, it is amplified by the strength of the transceiver. The signal hence reaches the vertex $x$ at level $k$ if and only if
$$
\sum_{i=0}^nR_{x_i}>\sum_{i=0}^{n}C_{(x_i,x_{i+1})}\quad\mbox{for all }n=0,1,\ldots,k-1.
$$
Write $O_{\sss{\rm comp}}$, $O_{\sss{\rm bond}}$ and $O_{\sss{\rm aug}}$ for the sets of vertices that are reached by a signal transmitted by the root using complete routing, boundary routing and augmented routing, respectively. Clearly, complete routing dominates boundary routing in the sense that $O_{\sss{\rm bond}}\subset O_{\sss{\rm comp}}$. Furthermore, augmented routing dominates complete routing in the same sense. Indeed, with augmented routing, the strength of a transceiver may be stored and used at any point in the forwarding process, while in complete routing, a transceiver at $x$ is only effective within $\Lambda_x$. Hence,
\begin{equation}\label{eq:hierarchy}
O_{\sss{\rm bond}}\subseteq O_{\sss{\rm comp}}\subseteq O_{\sss{\rm aug}}\quad \mbox{a.s.}
\end{equation}
Note that, if $R\geq C$ almost surely, then all three schemes can trivially transmit indefinitely, while on the other hand, if $R<C$ almost surely, then a signal has no chance of spreading at all in any of the schemes. Hence the interesting case is when $\{R\geq C\}$ has a non-trivial probability. It is then natural to investigate the possibility of infinite transmission in the schemes and to compare the schemes in this sense. Are there for instance cases when complete routing (and thereby also augmented routing) can transmit indefinitely but not boundary routing? And are there cases when augmented routing, but not boundary routing and complete routing, can transmit indefinitely? Furthermore, one might ask in general what happens when one or both of the variables $R$ and $C$ have power-law distributions. For what values of the exponents is it possible to transmit a signal over arbitrarily large distances? These questions can, and will, be answered by analyzing the conditions for infinite output range derived below.

The paper is organized so that augmented routing is analyzed in Section 3, using tools related to branching random walks. Complete routing and boundary routing are then treated in Section 3 and 4, respectively, by generalizing the arguments from Section 3. In each section, we also give examples and make the conditions more explicit for certain distribution types. Section 5 contains a summary, further comparison of the derived conditions and some directions for further work. Throughout we assume that $\{R\geq C\}$ has non-trivial probability.

\subsection{Related work}

Probability on trees has been a very active field of probability for the last decades; see e.g.\ \cite{climb} for an introduction and \cite{LyonsPeresbook} for a recent account. The work here is closely related to first passage percolation on trees and tree-indexed Markov chains, see e.g.\ \cite{benjaminiperes94perc, lyonspem}. We also rely on results and techniques for branching random walks, see \cite{Zhan}. Transceiver networks have previously been analyzed in the probability literature in the context of spatial Poisson processes, see \cite{boll}, but the setup there is quite different from ours.

\section{Augmented routing}

We begin by analyzing the augmented routing scheme. To this end, first note that the transmission process can be represented by a process that we will identify below as a killed branching random walk: Define $V_0=0$ for the root and then, for a vertex $y$ that is a child of $x$, let $V_y=V_x+Z_{x,y}$, where $Z_{x,y}=R_x-C_{(x,y)}$. This means that $V_y$ keeps track of the strength of the signal when it arrives at $y$. When $V_y$ takes on a negative value, the process dies at that location and the subtree below $y$ is declared dead.

If $m=1$, we have a random walk, killed when it takes a negative value. Hence, in this case,
$\PP\left(|O_{\sss{\rm aug}}|=\infty\right) > 0$ if $\E[R] > \E[C]$ and
$\PP\left(|O_{\sss{\rm aug}}|=\infty\right) = 0$ if $\E[R] \leq \E[C]$ and both expectations are finite. If $R$ and $C$ both have infinite expectations, both scenarios can happen. For the remainder of the section we assume that $m\geq 2$.

A one-dimensional, discrete-time branching random walk may be defined as follows: At the beginning, there is a single particle located at $V_0=0$. Its children, who form the first generation, are positioned according to a certain point process. Each of the particles in the first generation gives birth to new particles that are positioned (with respect to their birth places) according to the same point process; they form the second generation. The system goes then on according to the same mechanism. See for instance \cite{Zhan} for an account of results on this model.

In our case, each particle has $m$ children and the point process of displacements of the children of $x$ consists of $\{Z_{x,y}: y \hbox{ child of } x\}$. Let $\mathcal{V}$ denote the vertex set of the tree. The process starts with $V_0=0$ and, for a vertex $y$ that is a child of $x$, we have $V_y=V_x+Z_{x,y}$, where $\{Z_{x,y}: y \hbox{ child of } x\}_{x \in \mathcal{V}}$ form a collection of i.i.d.\ random variables. Note that, unlike in ``classical'' branching random walk, the displacements $\{Z_{x,y}: y \hbox{ child of } x\}$ are not i.i.d., since, for a fixed $x$, the term $R_x$ appearing in the definition of $Z_{x,y}$ is the same for all children of $x$. Nevertheless, $\{Z_{x,y}: y \hbox{ child of } x\}_{x \in \mathcal{V}}$ are i.i.d.\  and hence $\{V_y\}$ fits in the more general definition of a branching random walk above.

Now kill the branching random walk at $0$, that is, whenever $V_x < 0$, the process dies and the subtree below the vertex $x$ is declared dead. The survival probability in this killed random walk coincides with the probability of infinite transmission for augmented routing, and we would hence like to obtain a condition that determines when the survival probability is strictly positive. To this end, let $Z, Z_1, Z_2,\ldots$ be i.i.d.\ with the same law as $Z_{x,y}$ and let $I(\cdot)$ be the right-side large deviation rate function for $Z$, defined by
\begin{equation}\label{rsrate}
I(s): = \sup\limits_{\lambda \geq 0}[\lambda s - \log E[\exp(\lambda Z)]]\in [0, \infty]\, .
\end{equation}
Then Cram\'er's Theorem implies that
\begin{equation}\label{ld}
\lim\limits_{n \to \infty}\frac{1}{n}\log \PP\left[\frac{Z_1 + \cdots + Z_n}{n} \geq s\right] = -I(s),
\end{equation}
see \cite[Theorem 2.2.3]{dembozeitouni}, and hence $I(s)$ describes deviations ``to the right'' of $s$ (note that $\lambda$ is only running through the non-negative reals). In particular, we have $I(s) =0$ if $s \leq \E[Z]$.

Define
$$
s^* : = \sup\{s: I(s) \leq \log m\}\in (-\infty, \infty]\, .
$$
Note that, since $I(\cdot)$ is convex and non-decreasing, with $I(s) =0$ for $s \leq \E[Z]$, we have that
$s^* > 0$ if and only if $I(0) < \log m$. With this at hand, we can determine when the killed branching random walk which describes the transmission process with augmented routing has a strictly positive survival probability.

\begin{prop}\label{charsurvival}
Let $m \geq 2$. For the survival probability $\alpha=\PP\left(|O_{\sss{\rm bond}}|=\infty\right)$ of the killed branching random walk, we have $\alpha > 0$ if and only if $s^* >  0$. In particular, $\alpha > 0$ if and only if $I(0) < \log m$. Note that, if $\E[R - C] \geq 0$, then $I(0) = 0$ so that $\alpha > 0$.
\end{prop}

\begin{corr}
If $m \geq 2$, then $\PP\left(|O_{\sss{\rm aug}}|=\infty\right) > 0$ if and only if
\begin{equation}\label{critaug}
\E[e^{\lambda R}]\cdot \E[e^{-\lambda C}] >  \frac{1}{m}\quad \hbox{ for all } \lambda \geq 0\,.
\end{equation}
\end{corr}

The proposition morally follows from Theorem \ref{speedBRW} below, which goes back to J.\ D.\ Biggins, J.\ M.\ Hammersley, J.\ F.\ C.\ Kingman, see \cite{biggins,hammersley,kingman}. For a proof, we also refer to \cite[Theorem 2.1]{Zhan}. However, we will not need Theorem \ref{speedBRW}, but will give a direct proof of Proposition \ref{charsurvival} that we will then apply also for complete routing and boundary routing.

\begin{theorem}[Biggins, Hammersley, Kingman] \label{speedBRW}
For a branching random walk $\{V_x\}$, we have that
\begin{equation}\label{speed}
\lim\limits_{n \to \infty}\frac{1}{n}\max\limits_{x \in \mathcal{V}, |x|=n} V_x = s^*\quad \PP- \rm{a.s.}
\end{equation}
\end{theorem}

\noindent \emph{Proof of Proposition \ref{charsurvival}.} The proof is based on two standard arguments, which we recall since we will use them later. We also refer to \cite{climb}. We first show that the survival probability is 0 if $s^*<0$ by showing that
\begin{equation}\label{UBspeed}
\limsup_{k}\frac{1}{k}\max_{x:|x|=k}V_x\leq  s^*\quad \PP- \rm{a.s.}
\end{equation}
Indeed, (\ref{UBspeed}) implies that, if the branching random walk is killed at the ``linear barrier'' $sk$ with $s > s^*$ (i.e.\ all vertices $x_k$ at distance $k$ from the root with $V_{x_k} < s k$ are removed along with all their descendants), then it will die out almost surely. Our process is killed at $s=0$ and hence $\alpha =0$ if $s^* < 0$.

To establish (\ref{UBspeed}), we will consider the probabilities that there is a vertex $x$ at distance $k$ from the root with $V_x\geq s^* + \delta$, and use a union bound. There are $m^k$ such vertices, and if $\PP(V_x\geq s^* + \delta)$ decays fast enough, our probabilities will be summable. Assume that $s^* < \infty$ and fix $\delta > 0$. Then there is $\varepsilon > 0$ such that $I(s^*+ \delta) - \varepsilon > \log m$. Take $k$ large enough such that
$$
\PP\left[\frac{Z_1 + \cdots + Z_k}{k} \geq s^* + \delta \right] \leq \exp(-k (I(s^* + \delta )- \varepsilon))\, .
$$
Now, by a union bound,
$$
\PP\left[\frac{1}{k}\max_{x:|x|=k}V_x \geq s^* + \delta \right] \leq m^k P\left[\frac{Z_1 + \cdots + Z_k}{k} \geq s^* + \delta \right]
$$
$$
\leq m^k \exp(-k (I(s^* + \delta )- \varepsilon))
$$
and we conclude, using the Borel-Cantelli lemma, that
$$
\limsup_{k}\frac{1}{k}\max_{x:|x|=k}V_x\leq  s^* + \delta \quad \PP- \rm{a.s.}
$$
Since $\delta > 0$ was arbitrary, \eqref{UBspeed} follows from this.

To show that the survival probability is strictly positive if $s^* > 0$, we will construct a supercritical Galton-Watson process embedded in our tree. To this end, first note that
$$
\lim\limits_{n \to \infty}\frac{1}{n}\log \PP\left[\frac{Z_1 + \cdots + Z_j}{j} \geq s, j = 0,1, \ldots ,n \right] = -I(s),
$$
see \cite{mogulskii} or \cite[Theorem 5.1.2]{dembozeitouni}. Fix $s<s^*$. Since $I$ is a convex function which is strictly convex on $\{x: I(x) \in (0, \infty)\}$, we can pick $\delta > 0$ such that $I(s) < \log m -\delta$. Consider an embedded Galton-Watson process consisting of all vertices at distances $k, 2k, 3k, \ldots$ from the root such that the path of the branching random walk between the vertex (at distance $ik$ from the root, say) and its predecessor (at distance $(i-1)k$ from the root) stays strictly above $\ell s$ at distance $\ell= (i-1)k +j$ $(j=0,1, \ldots, k)$ from the root. Take $k$ large enough such that
$$
\PP\left[\frac{Z_1 + \cdots + Z_j}{j} \geq s, j =0,1, \ldots ,k \right]\geq \exp(-k(I(s) + \delta)).
$$
Then the embedded Galton-Watson process has expected offspring at least $\exp(-k(I(s)+ \delta))m^k > 1$, and therefore it has a strictly positive survival probability. An infinite path $0=x_0,x_1,x_2\ldots$ from the root, where $x_i$ is a child of $x_{i-1}$, for all $i$, is called a ray. The above argument shows that for $s < s^*$, we have that
\begin{equation}\label{LBspeedinf}
\PP\left(\exists\mbox{ a ray }\{x_n\}\mbox{ with } V_{x_n}\geq ns\mbox{ for all }n\right)>0.
\end{equation}
In particular, if the branching random walk is killed at the ``linear barrier'' $s k$, with $s < s^*$, it survives with positive probability. Hence, $\alpha > 0$ if $s^* > 0$, since our process is killed at $s=0$.

Finally we consider the critical case $s^* =0$. This requires a refinement of the argument for the case when $s^*>0$: Assume that $s^* =0$, so that $I(0) = \log m$. Then, by the Bahadur-Rao Theorem (see \cite[Theorem 3.7.4]{dembozeitouni}), there is a constant $c< 0$ such that, for all $k$,
$$
\PP\left[\frac{Z_1 + \cdots + Z_k}{k} \geq 0 \right] \leq \frac{c}{\sqrt{k}} \exp(-k (I(0))\, .
$$
Now consider the probability that there is a vertex $x$ at distance $k$ from the root with $V_x\geq 0$. There are $m^k$ such vertices and, using a union bound, we get that
$$
\PP\left[\frac{1}{k}\max_{x:|x|=k}V_x \geq 0\right] \leq m^k P\left[\frac{Z_1 + \cdots + Z_k}{k} \geq 0 \right]\leq \frac{c}{\sqrt{k}}.
$$
We conclude, using the Borel-Cantelli lemma along the subsequence $k^4$ $(k=1, 2, \ldots )$, that
$$
\frac{1}{k^4}\max_{x:|x|=k^4 }V_x \geq  0  \mbox{ only for finitely many $k$}, \quad \PP- \rm{a.s.}
$$
This implies that $\alpha =0$. In fact, much more is known:  A ``nearly optimal'' ray consists of vertices $x_k$ with $V_{x_k} \geq (s^* -\varepsilon)k$ for all $k$ and, in \cite[Theorem 1.2]{nyzsurvival}, it is shown that the probability that a nearly optimal ray exists goes to $0$ as $\varepsilon \to 0$.\hfill$\Box$

\begin{remark} Let $\partial \mathcal{T}$ denote the boundary of the tree which is defined as the set of all rays in the tree. One can use a $0-1$ law as in \cite[Proposition 3.2]{climb} to conclude from \eqref{LBspeedinf} that for $s < s^*$, we have
\begin{equation}
\PP(\sup\limits_{\xi \in \partial \mathcal{T}}\liminf_{x_k \in \xi, |x_k| = k}\frac{1}{k}V_{x_k}\geq  s) =1,
\end{equation}
which implies that
\begin{equation}\label{LBspeedsure}
\PP(\sup\limits_{\xi \in \partial \mathcal{T}}\liminf_{x_k \in \xi, |x_k| = k}\frac{1}{k}V_{x_k}\geq  s^*) =1\, .
\end{equation}
Now, Theorem \ref{speedBRW} follows, in our setup, from \eqref{UBspeed} and \eqref{LBspeedsure}.
\end{remark}

\textbf{Example 2.1.} Let $R$ and $C$ be Poisson distributed with mean $\mu_R$ and $\mu_C$, respectively. Then
$$
\log \E[\exp{\lambda(R-C)}]= (e^\lambda -1)\mu_R + (e^{ - \lambda} -1)\mu_C
$$
and \eqref{critaug} yields, after an easy calculation that infinite transmission is possible if and only if $\sqrt{\mu_C} - \sqrt{\mu_R} < \sqrt{\log m}$.\hfill$\Box$

\textbf{Example 2.2.} Let $C\equiv 1$ and assume that a transceiver is either functioning with range 1 (probability $r_1\neq 1$) or non-functioning with range 0 (probability $r_0$). Then
$$
\E[e^{\lambda R}]\E[e^{-\lambda C}]=r_0e^{-\lambda}+r_1
$$
and we see that \eqref{critaug} is satisfied if and only if $r_1>1/m$. Next assume that a functioning transceiver has range 2 (probability $r_2=1-r_0$). We apply Proposition \ref{charsurvival}, calculating $I(0)=0$ if $r_2 \geq 1/2$ and $I(0)= -\log(2 \sqrt{r_2(1- r_2)})$ otherwise, and get that either $r_2 \geq 1/2$ or $r_2(1- r_2) > (4m^2) ^{-1}$. Hence infinite transmission is possible if and only if
$$
r_2 > \frac{1}{2}\left(1 - \sqrt{1-1/m^2}\right)\, .
$$
\hfill$\Box$

\textbf{Example 2.3.} Consider the case with $R\equiv 1$ and $C\in\{0,2\}$, with $\PP(C=0)=p_0$ and $\PP(C=2)=p_2$. This is equivalent to the previous example with $r_2=p_0$ and $r_0=p_2$ in the sense that the effect of passing a transceiver and a consecutive edge is that either the signal strength is increased by 1 (probability $p_0$) or decreased by 1 (probability $p_2$). It follows that infinite transmission is possible if and only if
$$
p_0 > \frac{1}{2}\left(1 - \sqrt{1-1/m^2}\right)\, .
$$
\hfill$\Box$

\textbf{Example 2.4.} Set $C\equiv 1$ and let $R$ be Poisson distributed with mean $\mu$. Then (\ref{critaug}) is equivalent to
\begin{equation}\label{eq:1po_aug}
(e^\lambda-1)\mu-\lambda+\log m > 0\quad \mbox{for all }\lambda\geq 0.
\end{equation}
The minimal value is attained for $\lambda=-\log \mu$, and hence (\ref{eq:1po_aug}) is true if $\mu > 1$ or $1-\mu+\log(m\mu) > 0$. For $m=2$, we obtain numerically that infinite transmission is possible if and only if $\mu > 0.23$.\hfill$\Box$

\section{Complete routing}

For complete routing, the transmission process can be described by a process $\{W_y\}$ that keeps track of the remaining range of a signal from the root when it reaches $y$ and is defined as follows: Set $W_0=0$ for the root and then, for a vertex $y$ that is a child of $x$ in the tree, let
$$
W_y = \left\{ \begin{array}{ll}
                      R_x-C_{(x,y)} & \mbox{if $R_x>W_x$};\\
                      W_x-C_{(x,y)} & \mbox{otherwise}.
                    \end{array}
            \right.
$$
Indeed, if $R_x>W_x$, then the range of the transceiver at $x$ is larger than the remaining range of the routed signal at $x$. Hence, by the definition of the scheme, the remaining range at a given child $y$ of $x$ is $R_x$ minus the cost $C_{(x,y)}$ of the edge $(x,y)$. If $R_x\leq W_x$ on the other hand, then the transceiver at $x$ does not increase the remaining range, and the remaining range at a given child $y$ is therefore $W_x$ minus the cost $C_{(x,y)}$ of the edge $(x,y)$. When $W_y$ takes on a negative value, the process dies at that location and all vertices in the subtree below $y$ are assigned the value $\Gamma$, where $\Gamma$ is a cemetery state.

If $m=1$, we have a Markov chain, killed when it takes a negative value. When $R$ has bounded support, say $R\leq b$ almost surely, then infinite transmission is not possible: Let $c>0$ and $r<c$ be any numbers such that $\PP(C\geq c)>0$ and $\PP(R\leq r) > 0$. Consider a sequence of length $\lceil b/(c-r)\rceil+1$ such that the strength of each transceiver is at most $r$ while the cost of the incoming link is at least $c$. Such a sequence occurs eventually with probability 1 and, since $W_y\leq b$, it is not hard to see that it kills the signal. Furthermore, one sees directly, or from (\ref{eq:hierarchy}), that infinite transmission is not possible if $\E[R]\leq \E[C] < \infty$. In the general case, we do not know if survival is possible.

Assume for the remainder of the section that $m \geq 2$. The process $\{W_y\}$ is not a branching random walk. It is also not a tree-valued Markov chain in the sense of \cite{benjaminiperes94}, since the values of the vertices of two children of $x$ are not chosen independently given $W_x$. In addition, the Markov process we are considering is not irreducible. Nevertheless, the arguments of the previous section apply and we can give conditions for a positive survival probability. To this end, let $W_0, W_1, W_2, \ldots $ be a Markov process with the same law as $W_0, W_{x_1}, W_{x_2}, \ldots$, where $x_i$ is a child of $x_{i-1}$. Hence, the transition mechanism is the following: Take two i.i.d.\ sequences $\{C_i\}$ and $\{R_i\}$ which are independent. Given $W_{i-1}$, we set $W_i=\Gamma $ if $W_{i-1} = \Gamma$, and if  $W_{i-1} \geq 0$, we set
\begin{eqnarray*}
W_i = \left\{ \begin{array}{ll}
                      R_i-C_i & \mbox{ if $R_i>W_{i-1}$ and $R_i - C_i \geq 0$};\\
                      W_{i-1}-C_i &  \mbox{ if $R_i\leq W_{i-1}$ and $W_{i-1}-C_i \geq 0$};\\
                      \Gamma & \mbox{ otherwise}.
                    \end{array}
            \right.
\end{eqnarray*}
Denote by $\PP_z$ the probability measure associated with the Markov process started from $z \in \RR$ (the transmission process is started from $W_0=0$ but in the proof of Theorem \ref{condcomplete} below we need to consider arbitrary starting points). This Markov process has $\Gamma$ as an absorbing state. Note that, due to subadditivity, the limit
$$
-\lim\limits_{n \to \infty}\frac{1}{n}\log \inf\limits_{z \in \RR^+} \PP_z(W_n \geq 0) =
-\lim\limits_{n \to \infty}\frac{1}{n}\log \inf\limits_{z \in \RR^+} \PP_z(W_n \neq \Gamma)
$$
exists. The state space of the Markov chain $\{W_i\}$ is (a subset of) $\{\Gamma\}\cup [0, \infty)$. We can think of the state space as an ordered set, with smallest element $\Gamma$, and we claim that
$$
\inf_{z\in\mathbb{R}^+}\PP_z(W_n\geq 0)=\PP_0(W_n\geq 0) \, .
$$
Indeed, using the natural coupling for two Markov chains distributed according to $\PP_z$ and  $\PP_y$, respectively, which is to take the same sequences $\{C_i\}$ and $\{R_i\}$ in the above construction, we see that for any $z$ and $y$ with $y< z$, the law of $W_1$ under $\PP_y$ is dominated by the law of $W_1$ under  $\PP_z$, and by induction, the law of $W_n$ under $\PP_0$ is dominated by the law of $W_1$ under  $\PP_z$, for any $z > 0$ and any $n$. We conclude that
\begin{equation}\label{beta_def}
\beta: = -\lim\limits_{n \to \infty}\frac{1}{n}\log \PP_0(W_n \geq 0)
\end{equation}
exists and that
\begin{equation}\label{bISinf}
\beta =  -\lim\limits_{n \to \infty}\frac{1}{n}\log \inf\limits_{z \in \RR^+} \PP_z(W_n \geq 0)\, .
\end{equation}

The following theorem asserts that complete routing can transmit indefinitely if $\beta<\log m$ but not if $\beta>\log m$.

\begin{theorem}\label{condcomplete}
Assume that $m \geq 2$ and let $\beta$ be defined as in \eqref{beta_def}.
\begin{itemize}
 \item[\rm{(i)}]If, for some subsequence $n_k$ of the integers with $n_k \to \infty$ as $k \to \infty$,
\begin{equation}\label{conddie}
\sum\limits_{k=1} ^\infty m^{n_k} \PP_0(W_{n_k} \geq 0) < \infty
\end{equation}
then $\PP\left(|O_{\sss{\rm comp}}|=\infty\right) = 0$. In particular, \eqref{conddie} is satisfied if $\beta  >  \log m$.\\
\item[\rm{(ii)}] If $\beta < \log m$, then $\PP\left(|O_{\sss{\rm comp}}|=\infty\right) > 0$.\\
\end{itemize}
\end{theorem}

\begin{proof}
The proof of (i) is the same as the proof of \eqref{UBspeed}, and the proof of (ii) is the same as the proof of \eqref{LBspeedinf}. Indeed, using a union bound,
$$
\PP\left[\frac{1}{k}\max_{x:|x|=k}W_x \geq 0\right] \leq m^k \PP\left[W_k\geq 0\right]
$$
and hence it follows from the Borel-Cantelli lemma that, if \eqref{conddie} holds, then
$$
\limsup_{k}\frac{1}{k}\max_{x:|x|=k}V_x< 0 \quad \PP- \rm{a.s.}
$$
Part (i) follows from this by the same argument as in the proof of \eqref{UBspeed}.

To show (ii), we again construct an embedded Galton-Watson tree which survives with positive probability. Pick $\delta>0$ such that $\beta<\log m -\delta$ and choose $k$ large enough such that
$\inf\limits_{z \in \RR^+} \PP_z(W_k \neq \Gamma)\geq \exp(-k(\beta+\delta))$ (which is possible due to \eqref{bISinf}). Consider an embedded Galton-Watson process consisting of all vertices at distances $k, 2k, 3k, \ldots$ from the root such that the path of the branching random walk between the vertex (at distance $ik$ from the root, say) and its predecessor (at distance $(i-1)k$ from the root) does not hit $\Gamma$ (note that it suffices that $W$ takes non-negative values at the vertex and its predecessor). Then, the embedded Galton-Watson process has expected offspring at least $ \exp(-k(\beta+ \delta))m^k > 1$, and therefore it has a strictly positive survival probability.
\end{proof}

Obtaining explicit expressions for the probability $\PP(W_n\geq 0)$, and thereby for $\beta$, for some large class of distributions seems difficult. However, it is possible to deduce from Theorem \ref{condcomplete} that infinite transmission is always possible when $R$ is a power-law or when $\PP(C=0)>1/m$. The (simple) proofs of this are valid also for boundary routing, and therefore we give the proofs in the next section, see Corollary \ref{cor:pl} and \ref{cor:c_atom}. Here instead we analyze the condition in Theorem \ref{condcomplete} for some specific examples.

\textbf{Example 3.1} Let $C\equiv 1$ and $R\in\{0,1\}$ with $\PP(R=1)=r_1$. Then $W_n\geq 0$ if and only if no transceiver up to vertex $n$ is non-functioning with range 0. Hence $\PP(W_n\geq 0)=r_1^n$, so that $\beta=-\log r_1$, which is smaller than $\log m$ when $r_1>1/m$. This is the same condition as in Example 2.1, and indeed all schemes are equivalent in this case.\hfill$\Box$

\textbf{Example 3.2} Let $C\equiv 1$ and $R\in\{0,2\}$ with $\PP(R=0)=r_0$ and $\PP(R=2)=r_2=: r$.
In this case, $W_0, W_1, W_2, \ldots $ is a Markov chain with state space $\{\Gamma, 0, 1 \}$ and
with transition probabilities given by $p(\Gamma, \Gamma) =1$, $p(0, \Gamma) =r_0$, $p(0, 1) =r$, $p(1, \Gamma) =0$, $p(1, 0) = r_0$,  $p(1, 1) =r$. The transition matrix can be diagonalized, and has the eigenvalues $1, \frac{1}{2}(r + a)$ and $\frac{1}{2}(r - a)$, where $a = \sqrt{4r - 3 r ^2}$. We conclude that $\beta = -\log(\frac{1}{2}(r + a))$. Hence
$$
\PP\left(|O_{\sss{\rm comp}}|=\infty\right) > 0 \hbox{ if } r + \sqrt{4r - 3 r ^2} > \frac{2}{m}\, .
$$
and
$$
\PP\left(|O_{\sss{\rm comp}}|=\infty\right) = 0 \hbox{ if } r + \sqrt{4r - 3 r ^2} < \frac{2}{m}\, .
$$
This can be rewritten as
$$
\PP\left(|O_{\sss{\rm comp}}|=\infty\right) > 0 \hbox{ if } r  > \frac{1}{2}\left(1 + \frac{1}{m} - \sqrt{1+ \frac{2}{m} - \frac{3}{m^2}}\right)
$$
and
$$
\PP\left(|O_{\sss{\rm comp}}|=\infty\right) = 0 \hbox{ if } r  <  \frac{1}{2}\left(1 + \frac{1}{m} - \sqrt{1+ \frac{2}{m} - \frac{3}{m^2}}\right).
$$
In particular, recalling the condition for augmented routing from Example 2.3, we see that $r$ can be chosen such that infinite transmission is possible for augmented routing, but not for complete routing. For $m=2$ for instance, the critical value for $r$ is approximately $0.19$ with complete routing and approximately $0.067$ with augmented routing. We remark that, diagonalizing the transition matrix, one can compute that
$$
\PP(W_n \geq 0) = \frac{r + a}{2a}\left(\frac{r + a}{2}\right)^n + \frac{r - a}{2a}\left(\frac{r - a}{2}\right)^n ,
$$
but this does not help to settle the critical case $\beta = \log m$, since \eqref{conddie} is not satisfied. In general, we believe that, in the critical case, both scenarios are possible depending on the distributions.\hfill$\Box$

\textbf{Example 3.3} Next, we give another example where augmented routing is strictly more powerful than complete routing. To this end, recall Example 2.2, where is was shown that, when $R\equiv 1$ and $C\in\{0,2\}$, with $\PP(C=0)=p_0$, then infinite transmission is possible with augmented routing if and only if $p_0 > \frac{1}{2}(1 - \sqrt{1-1/m^2})$. For complete routing we note that $W_n\geq 0$ if and only if no edge between the root and vertex $n$ has weight 2. Thus $\beta=\log p_0$, implying that infinite transmission is possible if $p_0>1/m$, but not if $p_0<1/m$. For $m \geq 2$ and $p_0 = \frac{1}{2m}$, augmented routing can hence transmit indefinitely, but complete routing cannot.\hfill$\Box$

\section{Boundary routing}

First note that, when $\{R\geq C\}$ has a non-trivial probability, infinite transmission is never possible with boundary routing for $m=1$. Indeed, the tree $\mathcal{T}$ then reduces to a singly infinite path and the time until we encounter a transceiver at the boundary of the set of the informed vertices whose strength is strictly smaller than the cost of the edge to its un-informed neighbor is clearly almost surely finite. We hence restrict to $m\geq 2$.

We begin by giving an explicit condition for infinite transmission in the case when $C\equiv c$. By scaling we can take $c=1$ and it is then enough to consider integer-valued range variables $R$. Indeed, if $R$ is not integer-valued we instead work with $R'=\lfloor R\rfloor$ and note that this gives rise to the same transmission process.

\begin{prop} If $C\equiv 1$ and $R$ is integer-valued with $\PP(R=i)=r_i$ ($i= 0,1,2, \ldots$), then $\PP\left(|O_{\sss{\rm bond}}|=\infty\right)>0$ if and only if
\begin{equation}\label{eq:bond_cond_const}
\E\left[m^R\right]>1+r_0.
\end{equation}
\end{prop}

\begin{proof}
The condition follows by relating the transmission process to a branching process: The ancestor of the process is the root 0, and the offspring of a vertex $x$ then is $\partial \Lambda_x$, that is, the vertices that are within the range of $x$, but that have at least one child that is not within the range of $x$. The possible offspring of $x$ are the vertices at level $R_x$ below $x$, and since there are $m^k$ vertices at level $k$ below $x$, the offspring mean is $\sum_{k=1}^\infty m^kr_k=\E[m^R]-r_0$.
\end{proof}

\textbf{Example 4.1.} Let $C\equiv 1$ and assume that a transceiver is either functioning with range $n$ (probability $r_n\neq 1$) or non-functioning with range 0 (probability $r_0=1-r_n$). The root can then transmit indefinitely if and only if $r_n>1/m^n$. For $n=2$, the condition becomes $r_2>1/m^2$, which is strictly stronger than the condition for complete routing derived in Example 3.2. The critical value for $r_2$ when $m=2$ for instance is 0.25 with boundary routing and approximately 0.19 with complete routing. If $r_i=L(i)a^{-i}$ for some slowly varying function $L(i)$ and $a<1$, then infinite transmission is possible for $a>1/m$, while for $a<1/m$ it depends on the precise form of the distribution.\hfill$\Box$

\textbf{Example 4.2.} Take $C\equiv 1$ and let $R$ be Poisson distributed with mean $\gamma$. Then (\ref{eq:bond_cond_const}) translates into
$$
e^{\gamma(m-1)}>1+e^{-\gamma},
$$
which holds for $\gamma$ large enough. For $m=2$, the threshold is $\gamma=\ln(1+\sqrt{2})=0.88$. This can be compared to the condition for augmented routing, which is $\gamma > 0.23$, see Example 2.4. For larger $m$, analytical expressions for the threshold are more involved, but numerical values are easily obtained.\hfill$\Box$

When the edge costs are random, a branching process approach does not work, since information on that the signal has reached a vertex $x$, but not a given child $y$, affects the distribution of $C_{(x,y)}$ in a way that is difficult to control. Also the number of un-informed children of $x$ carries information about $C_{(x,y)}$. However, the arguments from the previous section can be applied again to derive a general condition. To this end, we note that the transmission process can be described by a process $\{U_y\}$ that keeps track of the strength of a signal from the root when it reaches $y$ and is defined as follows: Set $U_0=0$ for the root and then, for a vertex $y$ that is a child of $x$ in the tree, let
$$
U_y = \left\{ \begin{array}{ll}
                       U_x-C_{(x,y)}& \mbox{if $U_x-C_{(x,y)}\geq 0$};\\
                       R_x-C_{(x,y)}& \mbox{otherwise}.
                    \end{array}
            \right.
$$
Indeed, when $U_x-C_{(x,y)}$ becomes strictly negative, we have passed a vertex that is on the boundary of the informed set. The transceiver at $x$ then forwards the signal and the new balance is $R_x-C_{(x,y)}$. When $U_y$ takes on a negative value, the process dies at that location and all vertices in the subtree below $y$ are assigned the value $\Gamma$, where $\Gamma$ is a cemetery state.

Let $U_0,U_1,\ldots$ be a Markov process distributed as the above process along a given ray in the tree, that is, $\Gamma$ is an absorbing state and, if $U_{i-1}\geq 0$, the transition mechanism is
$$
U_i = \left\{ \begin{array}{ll}
                       U_{i-1}-C_i & \mbox{if $U_{i-1}-C_i\geq 0$};\\
                       R_{i-1}-C_i & \mbox{if $U_{i-1}-C_i<0$ and $R_{i-1}-C_i\geq 0$};\\
                       \Gamma & \mbox{otherwise.}
                    \end{array}
            \right.
$$
Here $\{R_i\}$ and $\{C_i\}$ are i.i.d.\ sequences. Let $\PP_z$ denote the probability measure of the process $\{U_i\}$ started from $U_0=z$. In analogy with complete routing, the limit
\begin{equation}\label{eq:lim_inf_bound}
-\lim\limits_{n \to \infty}\frac{1}{n}\log \inf_{z\in\mathbb{R}^+}\PP_z(U_n \geq 0)
\end{equation}
exists due to subadditivity. Furthermore, we have also in this case that
\begin{equation}\label{eq:temp}
\lim\limits_{n \to \infty}\frac{1}{n}\log \inf_{z\in\mathbb{R}^+}\PP_z(U_n \geq 0)=\lim\limits_{n \to \infty}\frac{1}{n}\log \PP_0(U_n\geq 0).
\end{equation}
Indeed, if the chain is started from $U_0=z>0$, for sure it survives to the level $M_z=\max\{k:\sum_{i=1}^kC_i\leq z\}$, and from that point the mechanism is stochastically the same as for a process started from $U_0=0$. Hence,
\begin{equation}\label{gamma_def}
\gamma :=  -\lim\limits_{n \to \infty}\frac{1}{n}\log \inf\limits_{z \in \RR^+} \PP_z(U_n \geq 0)\,
\end{equation}
exists and coincides with the limit in \eqref{eq:lim_inf_bound}. This means that the proof of Theorem \ref{condcomplete} goes through verbatim and gives an analogous criteria for infinite transmission with boundary routing.

\begin{theorem} Assume that $m \geq 2$ and let $\gamma$ be defined as in \eqref{gamma_def}.
\begin{itemize}
\item[\rm{(i)}] If $\gamma  >  \log m$, then $\PP\left(|O_{\sss{\rm bond}}|=\infty\right) = 0$.
\item[\rm{(ii)}] If $\gamma < \log m$, then $\PP\left(|O_{\sss{\rm bond}}|=\infty\right) > 0$.
\end{itemize}
\end{theorem}

Just as for complete routing, it is typically difficult to find explicit expressions for $\gamma$. However, in some cases we can give sufficient conditions for $\gamma<\log m$, and hence for the possibility of infinite transmission. First recall that a tail distribution function $\bar F(x)=\PP(X>x)$ is said to be regularly varying with tail exponent $\tau-1$ if $\bar F(x)=x^{-(\tau-1)}L(x)$, where $x\mapsto L(x)$ is slowly varying at infinity (that is, $L(ax)/L(x)\to 1$ as $x\to \infty$ for any $a>0$). When this is the case, we say that the random variable $X$ has a power-law distribution.

\begin{corr}\label{cor:pl}
If $m \geq 2$ and $R$ has a power-law distribution, then $\PP\left(|O_{\sss{\rm bond}}|=\infty\right)>0$ regardless of the distribution of $C$.
\end{corr}

\begin{proof}
Let $S_n=\sum_{i=1}^nC_i$. Trivially $\PP(U_n\geq 0)\geq \PP(S_n\leq R)$, since the process is clearly alive at level $n$ if the total cost of a given path of length $n$ does not exceed the range of the root transceiver. For any $c>0$, we have that $\PP(S_n\leq R)\geq \PP(R\geq nc)\cdot\PP(S_n\leq nc)$ and trivially $\PP(S_n\leq nc)\geq \PP(C\leq c)^n$. Now take $c$ such that $\PP(C\leq c)\geq a/m$ for some $a\in(1,m)$. Then
$$
\PP(S_n\leq R)\geq \PP(R\geq nc)\cdot (a/m)^n
$$
and it follows that $\gamma<\log m$.
\end{proof}

The tail behavior of the cost variable $C$ does not have the same role in determining the possibility of infinite transmission. For instance, it is not the case that infinite transmission is necessarily impossible if $C$ has a power-law distribution while $R$ has a distribution with an exponentially decaying tail. Instead, a sufficiently large atom at 0 for $C$ guarantees that infinite transmission is possible, regardless of the tail behavior of the distributions.

\begin{corr} \label{cor:c_atom}
If $\PP(C=0)\geq 1/m$, then $\PP\left(|O_{\sss{\rm bond}}|=\infty\right)>0$.
\end{corr}

\begin{proof}
For any fixed $r\geq 0$, we have that
$$
\PP(S_n\leq r)\geq \PP(\cap_{i=1}^n\{C_i\leq r/n\})\geq \PP(C=0)^n.
$$
Since $\PP(U_n\geq 0)\geq \PP(S_n\leq R)$, this implies that $\gamma<\log m$.
\end{proof}

\section{Summary and conclusions}

We have derived sharp conditions for infinite transmission in all three schemes. For $m\geq 2$, the conditions are as follows:

\begin{itemize}
\item[$\bullet$] Augmented routing can transmit indefinitely if and only if $\E[{\rm{e}}^{\lambda R}]\E[{\rm{e}}^{-\lambda C}] > 1/m$ for all
$\lambda\geq 0$.
\item[$\bullet$] Let $\{W_i\}$ represent the transmission process with complete routing along a given ray (see Section 4 for a precise definition) and define $\beta=-\lim_{n \to \infty}\frac{1}{n}\log P(W_n \geq 0)$. Complete routing can then transmit indefinitely if $\beta<\log m$, but not if $\beta>\log m$. At the critical point $\beta=\log m$, we believe that both scenarios are possible.
\item[$\bullet$] For $m\geq 2$, boundary routing can transmit indefinitely if $\gamma <\log m$ but not if $\gamma>\log m$, where $\gamma=-\lim_{n \to \infty}\frac{1}{n}\log \PP(U_n \geq 0)$ and $\{U_i\}$ represents the transmission along a given ray. When $C\equiv 1$ and $R$ is integer-valued with $\PP(R=i)=r_i$, the condition becomes $\E[m^R]>1+r_0$.
\end{itemize}

For $m =1$, augmented routing can transmit indefinitely if $\E[R] > \E[C]$ but not if $\E[R] \leq \E[C]$ and both expectations are finite. If $R$ and $C$ both have infinite expectations, both scenarios can happen. Complete routing cannot transmit indefinitely when $R$ has bounded support, but the general case is open. Boundary routing cannot transmit indefinitely for any distribution.

When $R$ has a power-law distribution and $m\geq 2$, infinite transmission is always possible with all three schemes; see Corollary \ref{cor:pl}. The tail behavior of $C$ does not play the same role, since, according to Corollary \ref{cor:c_atom}, a large enough atom at 0 guarantees that infinite transmission is possible with boundary routing (and thereby also with the other schemes), regardless of the tail behaviors.

We have given several examples of distributions where the schemes are strictly different in the sense that there are regimes for the parameters of the distributions of $R$ and $C$ where one (or two) of the schemes can transmit indefinitely, but not the other two (one), see e.g.\ Example 3.2, 3.3, 4.1 and 4.2. Complete routing and boundary routing are trivially equivalent in some cases, e.g.\ when $R$ is constant (see also Example 3.1). An interesting question is if the three schemes are always strictly different when this is not the case and when $R$ does not have a power-law distribution, that is, is it then always strictly easier to transmit to infinity with complete routing than with boundary routing, and strictly easier with augmented routing than with complete routing? Or are there cases when the conditions coincide for (at least) two of the schemes? Answering this is complicated by the fact that the conditions for complete routing and boundary routing are somewhat difficult to analyze, since the probabilities $\PP(W_n\geq 0)$ and $\PP(U_n\geq 0)$ are typically not easy to calculate. 
However, when $C\equiv 1$ and $R$ is integer-valued, the conditions for boundary routing and for augmented routing are explicit and an easier question is if there are families of distributions of $R$ for which these conditions coincide. Below we show that the answer is no.

The condition \eqref{eq:bond_cond_const} for infinite transmission with boundary routing means that $r_0$ has to be sufficiently small. We now show that, for any distribution that satisfies \eqref{eq:bond_cond_const}, it is possible to strictly increase $r_0$ and still be able to transmit indefinitely with augmented routing. To this end, let the distribution of $R$ be described by $\{r_i\}_{i=0}^\infty$, let $k=\min\{i:r_i>0\}$ and take $\varepsilon\in(0,r_k)$. Then define $R_\varepsilon$ by shifting mass $\varepsilon$ from $k$ to 0, that is, $R_\varepsilon$ has distribution
$$
\PP(R_\varepsilon=i) = \left\{ \begin{array}{ll}
                       r_k-\varepsilon & \mbox{if $i=k$};\\
                       r_0+\varepsilon & \mbox{if $i=0$};\\
                       r_i & \mbox{otherwise.}
                    \end{array}
            \right.
$$

\begin{prop}
Let $m\geq 2$, take $C\equiv 1$ and let $R$ be integer-valued such that \eqref{eq:bond_cond_const} holds. If $\varepsilon$ is sufficiently small, then $R_\varepsilon$ satisfies \eqref{critaug}.
\end{prop}

\begin{proof}
We have that
\begin{equation}\label{eq:aug_mod}
m\E[e^{\lambda R_{\varepsilon}}]\E[e^{-\lambda C}] \geq r_0me^{-\lambda}-me^{(k-1)\lambda}\varepsilon+m\sum_{i=1}^\infty r_ie^{(i-1)\lambda}.
\end{equation}
First assume that $e^\lambda\geq m$, and write $e^\lambda=m+c_\lambda$, where $c_\lambda>0$ and $c_\lambda\sim e^\lambda$ as $\lambda\to\infty$. Then we obtain for the last term in \eqref{eq:aug_mod} that
$$
m\sum_{i=1}^\infty r_ie^{(i-1)\lambda}\geq \sum_{i=1}^\infty r_im^i+mr_kc_\lambda^{k-1}>1+mr_kc_\lambda^{k-1},
$$
where the last inequality follows from \eqref{eq:bond_cond_const}. It follows that, if $\varepsilon$ is sufficiently small, then $m\E[e^{\lambda R_{\varepsilon}}]\E[e^{-\lambda C}]>1$ for all $\lambda$ such that $e^\lambda\geq m$. Next assume that $e^\lambda<m$ so that $e^{-\lambda}>1/m$. Trivially
$$
m\sum_{i=1}^\infty r_ie^{(i-1)\lambda}\geq m(1-r_0)
$$
and hence the right hand side of \eqref{eq:aug_mod} is bounded from below by
$$
r_0-me^{(k-1)\lambda}\varepsilon+m(1-r_0),
$$
which is larger than 1 for all $\lambda$ in the specified range if $\varepsilon$ is sufficiently small.
\end{proof}

A possible continuation of the current work would be to investigate time dynamics of the transmission schemes under various rules for the transmission times. Conditionally on that the signal does not die, what is the asymptotic speed of the transmission? Are there setups when a scheme has a very small (large) probability of transmitting to infinity, but where the speed of transmission conditionally on survival is large (small)? Yet another question to investigate is when the root can hear a signal from infinitely far away. Do the conditions on $R$ and $C$ for this coincide with the conditions for infinite transmission? For all three routing schemes, the probability that the root can transmit to a given vertex $x$ at level $n$ is of course the same as the probability that $x$ can transmit to the root. However, the dependence structure for the events $\{\mbox{root can transmit to vertex $i$ at level $n$}\}_{i=1}^{m^n}$ and $\{\mbox{vertex $i$ at level $n$ can transmit to the root}\}_{i=1}^{m^n}$ is different, and hence the conditions could possibly be different. In \cite{dousse}, this issue is analyzed for a related problem in the context of a spatial Poisson process.

{\bf Acknowledgement: } We thank Silke Rolles and Anita Winter for inviting us to the ``Women in Probability'' workshop taking place in July 2010 at Technische Universit\"at M\"unchen, where this work was initiated.


\begin{thebibliography}{99}
\baselineskip=14pt

\bibitem{benjaminiperes94}
     Benjamini, I.\ and Peres, Y. (1994).
     Markov chains indexed by trees.
      {\it Ann.\ Probab.} \textbf{22}, 219--243.

\bibitem{benjaminiperes94perc}
     Benjamini, I.\ and Peres, Y. (1994).
     Tree indexed random walks on groups and first passage percolation.
      {\it Probab.\ Theory\ Rel.\ Fields} \textbf{98}, 91--112.

\bibitem{biggins}
    Biggins, J.~D.\ (1976).
    The first- and last-birth problems for a
    multitype age-dependent branching process.
    {\it Adv.\ Appl.\ Probab.} {\bf 8}, 446--459.

\bibitem{boll}
    Balister, P.\, Bollobas, B.\ and Walters, M.\ (2009).
    Random transceiver networks.
    {\it Adv.\ Appl.\ Probab.} {\bf 41}, 323--343.

\bibitem{dembozeitouni}
    Dembo, A. and Zeitouni, O. (1998).
    Large deviations techniques and applications.
    Second edition. Applications of Mathematics {\bf 38},
    Springer.

\bibitem{dousse}
    Dousse, O.\ (2012).
    Percolation in directed random geometric graphs.
    {\it IEEE Int. Symp. On Inf. Theory Proceedings}, 601-605.

\bibitem{nyzsurvival}
      Gantert, N.\ Hu, Y.\ and Shi, Z. (2011).
      Asymptotics for the survival probability in a killed branching random walk.
     {\it Ann. Inst. Henri Poincar\'e Probab. Stat.} {\bf 47},
     111--129.

\bibitem{hammersley}
    Hammersley, J.~M.\ (1974).
    Postulates for subadditive processes.
    {\it Ann.\ Probab.} {\bf 2}, 652--680.

\bibitem{kingman}
    Kingman, J.~F.~C.\ (1975).
    The first birth problem for an age-dependent
    branching process.
    {\it Ann.\ Probab.} {\bf 3}, 790--801.

\bibitem{lyonspem}
    Lyons, R.\ and Pemantle, R.\ (1992).
    Random walks in a random environment and first passage percolation on trees.
    {\it Ann.\ Probab.} {\bf 20}, 125--136.

\bibitem{LyonsPeresbook}
Lyons, R.\ and Peres, Y. \
Probability on trees and networks.
\texttt{http://mypage.iu.edu/\~{}rdlyons/prbtree/prbtree.html}


\bibitem{mogulskii}
    Mogulskii, A.~A.\ (1976).
    Large deviations for trajectories of multi-dimensional random walks.
    {\it Theory Probab.\ Appl.} {\bf 21},
    300--315.


\bibitem{climb}
     Peres, Y.\ (1999).
     Probability on trees: an introductory climb.
     {\it Lecture notes in mathematics}  {\bf 1717}
     Springer, 193--280

\bibitem{Zhan}
     Shi, Z.\ (2011).
     Random Walks and Trees.
     {\it ESAIM Proc.} {\bf 31},
     1--39.


\end{thebibliography}
\end{document}